\documentclass[11pt]{amsart}

\usepackage{amsfonts,amsmath,amssymb}
\usepackage{enumerate}
\usepackage[pdftex,colorlinks]{hyperref}

\numberwithin{equation}{section}

\newtheorem{theorem}{Theorem}[section]
\newtheorem{lemma}[theorem]{Lemma}

\newtheorem{corollary}[theorem]{Corollary}
\newtheorem{remark}[theorem]{Remark}

\newcommand{\cl}[1]{\mathcal{#1}} 
\newcommand{\bb}[1]{\mathbb{#1}}
\newcommand{\sca}[1]{\left(#1\right)} 

\def\gl{\lambda}

\newcommand\supp{\mathop{\rm supp}}

\begin{document}

\title[Transference and preservation]{Synthetic properties of locally compact groups: 
preservation and transference}

\author{M. Anoussis,  G. K. Eleftherakis, A. Katavolos}

\address{M. Anoussis\\Department of Mathematics\\ University of the Aegean\\
832 00 Karlovassi, Greece }\email{mano@aegean.gr}

\address{G. K. Eleftherakis\\Department of Mathematics\\ Faculty of Sciences\\
University of Patras\\265 00 Patra, Greece }
\email{gelefth@math.upatras.gr}

\address{A. Katavolos\\Department of Mathematics\\
National and Kapodistrian University of Athens\\
1578 84 Athens, Greece }\email{akatavol@math.uoa.gr}


\begin{abstract}  
Using techniques from TRO equivalence of masa bimodules we prove various transference results: 
We show that when $\alpha$  is a group homomorphism which pushes forward the Haar 
measure of $G$ to 
a measure absolutely continuous with respect to the Haar measure on $H$, 
then $(\alpha\times\alpha)^{-1}$
preserves sets of compact operator synthesis, and conversely when  $\alpha$  is onto.
We also prove similar preservation results for operator Ditkin sets and operator M-sets, obtaining 
preservation results for M-sets as corollaries. 
Some of these results extend or complement existing results of 
Ludwig, Shulman, Todorov and Turowska.
\end{abstract}

\maketitle

\section{Introduction}\label{10000}

In this paper we study 
preservation  of synthetic sets, $M$-sets  and Ditkin sets by group 
homomorphisms between locally compact groups  both 
in the classical and in the operator sense.

In the early seventies,  Arveson  discovered important   connections between
spectral synthesis and  invariant subspace theory and established
precise links between notions in harmonic analysis and in operator algebra theory
\cite{arv}. 

The failure of spectral synthesis for a locally compact group $G$ can be formulated as the existence of 
distinct closed ideals of the Fourier algebra of $G$ having the same support. 

It  was Arvseon in his seminal paper \cite{arv} who used this phenomenon to produce the first example 
of two distinct weak-* closed subalgebras on an $L^2(X,\mu)$ space containing 
the multiplication algebras 
of $L^\infty(X,\mu)$ which have the same `support', a closed subspace of $X\times X$. 

This was later formalised, in the commutative case, by J. Froelich \cite{F} and, in the general 
locally compact case, by J. Ludwig and L. Turowska \cite{lt}.

The theory initiated by Arveson has been   significantly developed 
by many authors \cite{akt, eks, F, lt, st, stwt, SSTT, tt}. 
Our work is a contribution to this circle of ideas. 

Let $G$ and $H$ be locally compact  second countable groups with Haar measures $\mu $ 
and $\nu $ respectively and let $\theta: G \rightarrow H$ be  a continuous homomorphism.
To study preservation properties with respect to $\theta$, 
we shall use a method based on  TRO-equivalence. 
This notion was introduced in \cite{ele} by G.K. Eleftherakis, and used in 
\cite{ele2} for the study of
masa-bimodules and in \cite{elee} for the study of problems of operator synthesis. 
In sections \ref{ontocomp}, \ref{sec4}  and \ref{msets} 
we prove  preservation  results for  operator synthetic sets, operator Ditkin sets  
and operator $M$-sets. 
The main idea of our method is the following:  
to prove the preservation  of   a property  from $H$ to $G$, we first
 express  the property in terms of  appropriate masa-bimodules. 
Then we show that the masa-bimodules associated to $H$ are TRO-equivalent 
to the ones associated to $G$, and
hence we obtain the preservation of the property we are interested in.

In fact we use this method for more general maps between standard measure spaces.

The concept of reduced spectral synthesis as well as its operator theoretic analogue are
introduced and studied in \cite{stwt2}. 
In particular,  an inverse image theorem for sets of compact operator synthesis   
follows from  \cite[Theorem 4.7]{stwt2}
under the assumption that    $\theta$ has a special form.

In section \ref{ontocomp},  we prove Theorem \ref{49}, which implies an inverse image theorem 
for sets of compact operator synthesis  for general $\theta$ 
assuming that $\theta_*(\mu)$  is absolutely continuous with respect to $\nu$.

We also  show in section \ref{sec4} that if $\theta$ is onto then an $\omega$-closed set
$\kappa\subseteq H\times H$ (see below for these notions) 
is a set of (compact) operator synthesis  
if and only if $(\theta\times\theta)^{-1}(\kappa)$ 
is a set of (compact) operator synthesis.

Note that it follows from  \cite[Theorem 4.7]{st}  that if 
$\theta_*(\mu)$ is absolutely continuous with respect  to $\nu$
and $\kappa\subseteq H\times H$ is an 
$\omega$-closed set of  operator synthesis,  then  $(\theta\times\theta)^{-1}(k)$ 
is a set of  operator synthesis.

In section \ref{ditkin} we prove a preservation result for  operator Ditkin sets
under the assumption that   $\theta$ has open image  and  compact kernel.  

In section \ref{msets} we study preservation and transference properties of $M$-sets.
These were introduced for general locally compact groups by Bo\.{z}ejko in \cite{bo}.
Shulman, Todorov and Turowska in \cite{stwt}  introduced the notion of 
$M_1$ sets in locally compact groups   and studied transference and preservation 
properties of $M$-sets and $M_1$-sets  and their operator analogues, 
which were introduced in \cite{stwtint}. 
Our results in section \ref{msets} complement and improve
some of the results of \cite{stwt}. 

We show that  if $\theta$ is  onto, an $\omega$-closed set
$\kappa\subseteq H\times H$ is an operator $M$-set  (resp. $M_1$-set)  if and only if 
$(\theta\times\theta)^{-1}(\kappa)$ is an operator $M$-set  (resp. $M_1$-set).
As a corollary we obtain, for a closed set $E\subseteq H$, that 
 $\theta^{-1}(E)$ is an $M$-set (resp. $M_1$-set) if and only if $E$  is an $M$-set (resp. $M_1$-set).

Related results were obtained in \cite{stwt} 
under the  assumption that $\theta_*(\mu)$ has a Radon-Nikodym derivative with 
respect to $\mu$ which is $\mu$-a.e. finite, which in turn is equivalent to 
the compactness of $\ker\theta$ (see Corollary \ref{26}). 
\bigskip

We will need some preliminaries. For more details, see for example \cite{akt} or
 \cite[Section 4]{stwt2}. 
  
Recall that a standard Borel measure space   $(X,\mu)$ is a measure space such that $X$ is Borel 
isomorphic to a Borel subset of a Polish space and $\mu$ is a regular $\sigma$
Borel measure on $X$ \cite[12.5]{kechr}. 

If $(X,\mu)$ and $(Y,\nu)$ are standard measure spaces, 
a subset $\kappa\subseteq X\times Y$ is called {\it marginally null} if
$\kappa\subseteq (M\times Y)\cup(X\times N)$, 
where $M\subseteq X$ and $N\subseteq Y$ are null.
A subset $\kappa$ of $X\times Y$ is called {\it $\omega$-open} if it is
marginally equivalent to (i.e. differs by a marginally null set from) 
the union of a countable set of Borel rectangles.
The complements of $\omega$-open sets are called {\it $\omega$-closed}. 

If  $\theta: X\rightarrow Y$ is a Borel map and $\mu$ is a measure on $X$,  we
 denote by  $\theta _*(\mu )$ the measure on $Y$ defined 
by $\theta _*(\mu )(E)= \mu (\theta^{-1}(E ))$ for  every Borel set $E \subseteq Y.$
 
If $H_1$ and $ H_2$ are Hilbert spaces, we write  $\cl B(H_1,H_2)$ for the bounded linear 
operators from $H_1$ to $H_2$. When $H_1=L^2(X,\mu)$ and $H_2=L^2(Y,\nu)$, 
we will call a subspace  $\cl W\subseteq\cl B(H_1,H_2)$  
a {\it masa-bimodule} if $M_gTM_f \in \cl  W$  
for all $T\in \cl W$, $f\in L^{\infty}(X,\mu)$ and $g\in L^{\infty}(Y,\nu)$, where 
$M_g$ and $M_f$ denote the corresponding multiplication operators.
The w*-closed masa bimodule generated by a set $\cl U\subseteq \cl B(H_1,H_2)$ will be denoted
${\rm Bim}(\cl U)$.

We say that a measurable subset $\kappa\subseteq X\times Y$
{\it supports} an operator $T\in \cl B(H_1,H_2)$
(or that $T$ {\it is supported by} $\kappa$)
if $P(F)TP(E) = 0$ whenever the rectangle $E\times F$ is
marginally disjoint from $\kappa$, and write
$$
{M}_{\max}(\kappa) = 
\left\{T\in \cl B(H_1,H_2) : T \mbox{ is supported by } \kappa\right\}.
$$
For any subset ${\mathcal W}\subseteq \cl B(H_1,H_2)$, there exists a smallest
(up to marginal equivalence) $\omega$-closed set
$\supp(\cl W)\subseteq X\times Y$ which supports every operator $T\in{\mathcal W}$ \cite{eks}.

By \cite{arv} and \cite{st}, for any $\omega$-closed set $\kappa$, there exists a
weak* closed bimodule ${M}_{\min}(\kappa)$ such that, 
for every weak* closed masa bimodule  $ M\subseteq \cl B(H_1,H_2)$ 
with support marginally equivalent to $\kappa$ we have that
$${M}_{\min}(\kappa)\subseteq {M}\subseteq {M}_{\max}(\kappa).$$


\section{TRO-equivalence and preservation}\label{ontocomp}

The main results of this section are Theorem \ref{equiv} and Theorem \ref{48}, 
where we use techniques from the TRO equivalence developed in \cite{ele}. 
We use Theorem \ref{48} to extend an inverse image theorem for sets of compact operator synthesis
due to  Shulman, Todorov and Turowska \cite[Theorem 4.7]{stwt2}.
We will    apply Theorem \ref{equiv}   to preservation results for group homomorphisms in 
sections \ref{sec4} and \ref{msets}.

\smallskip
 We begin with a general Lemma which is  perhaps of independent interest. 

\smallskip
If $\cl U$ is a subspace of operators acting on a Hilbert space we denote by 
$\cl U\cap \cl K$ its subspace of compact operators.

Let $H, K$ be Hilbert spaces and  $\cl M$  a norm-closed linear subspace of $ \cl B(H, K).$ 
The space $\cl M$ is  called an  essential ternary ring of operators (TRO) if   
 $\cl M\cl M^*\cl M\subseteq \cl M$ (TRO property) and  
$I_{H}\in \overline{[\cl M^*\cl M]}^{w^*}$ and 
$I_{K}\in \overline{[\cl M\cl M^*]}^{w^*}$ (essentiality). 

\begin{lemma}\label{47} Let $H_i, K_i$ be Hilbert spaces and $\cl M_i\subseteq \cl B(H_i, K_i), i=1,2$ 
be essential ternary rings of operators. Let  $\cl U$ be a 
w*-closed subspace of $ \cl B(H_1, H_2)$. We set  
$$F(\cl U)= \overline{[\cl M_2 \cl U\cl M_1^*]}^{w^*} \quad\text{and }\quad 
F_0(\cl U\cap \cl K)=\overline{[\cl M_2(\cl U\cap \cl K)\cl M_1^*]}^{\|\cdot\|}.$$ 
Then $$F_0(\cl U\cap \cl K)=F(\cl U)\cap \cl K\, .$$
\end{lemma} 

\begin{proof} From \cite[Proposition 2.11]{ele} it follows that 
$$\cl U=\overline{[\cl M_2^*F(\cl U)\cl M_1]}^{w^*} $$
 and hence
$\cl M_2^*(F(\cl U)\cap \cl K)\cl M_1\subseteq \cl U\cap \cl K\, .$ Therefore 
\begin{equation}\label{****} 
\cl M_2\cl M_2^*(F(\cl U)\cap \cl K)\cl M_1\cl M_1^*\subseteq \cl M_2(\cl U\cap \cl K)\cl M_1^*\, .
\end{equation}
Since the $\cl M_i$ are essential TRO's, there exist \cite[Corollary 8.1.24]{bm}
nets of the form 
$$ x_\lambda =\sum_{i=1}^{s_\lambda }z_i^\lambda (z_i^\lambda )^*, \;\;z_i^\lambda \in \cl M_2,
\;\;\;y_\lambda =\sum_{i=1}^{t_\lambda }w_i^\lambda (w_i^\lambda )^*, \;\;w_i^\lambda \in \cl M_1,
\;\;\lambda \in \Lambda  
$$ 
such that $\|x_\lambda \|\leq 1$,  $\|y_\lambda \|\leq 1$ for all $\lambda$ and 
$$\mbox{SOT-}\lim x_\lambda =I_{K_2}, \;\;\mbox{SOT-}\lim y_\lambda =I_{K_1}.$$
If $x\in F(\cl U)\cap \cl K$ we have from (\ref{****})
$$x_\lambda xy_{\lambda^\prime }\in [\cl M_2(\cl U\cap \cl K)\cl M_1^*] ,\;\;\;
\text{for all }\; \;\lambda,\; \lambda ^\prime \;\in \Lambda .$$
Since $x$ is a compact operator the iterated norm limits 
$\lim_{\gl}(\lim_{\gl'}x_{\gl}xy_{\gl'})=x$ exist and we have 
$$x\in \overline{[\cl M_2(\cl U\cap \cl K)\cl M_1^*]}^{\|\cdot\|}. $$
We conclude that 
$F(\cl U)\cap \cl K\subseteq  \overline{[\cl M_2(\cl U\cap \cl K)\cl M_1^*]}^{\|\cdot\|}$ and since 
$\cl M_2(\cl U\cap \cl K)\cl M_1^*\subseteq F(\cl U)\cap \cl K$ we obtain
$$F(\cl U)\cap \cl K=  \overline{[\cl M_2(\cl U\cap \cl K)\cl M_1^*]}^{\|\cdot\|}=F_0(\cl U\cap \cl K). $$
\end{proof}

\bigskip

For the rest of this section  $(X_i,\mu _i), (Y_i,\nu _i)$, $i=1,2$ will denote  standard Borel  
measure spaces and  $\phi _i: X_i\rightarrow Y_i\, (i=1,2)$  will denote measurable maps
such that $(\phi _i)_*(\mu _i)\ll\nu _i.$ 

We define 
$$
\tilde \phi _i: L^\infty (Y_i)\rightarrow L^\infty (X_i), 
\;\;\;\tilde \phi _i(f)=f\circ \phi _i,\;\;\; i=1,2
$$ 
and the TRO's 
\begin{equation}\label{defN}
\cl N_i=\{T\in\cl B(L^2(Y_i),L^2(X_i)): TP_i(E )=Q_i(\phi _i^{-1}(E ))T, 
\;\;\text{for all }\; \; E\subseteq Y_i \;\;\text{Borel}\}
\end{equation}
(here $P_i(E )$ (resp. $Q_i( \phi _i^{-1}(E))$) is the projection onto 
$L^2(E)$ (resp. $L^2( \phi _i^{-1}(E))$)).
Fix Borel sets $E_i\subseteq Y_i, i=1,2$  such that 
$\ker \tilde{\phi _i}=L^\infty (Y_i\setminus E_i)$. Then the maps 
$$
L^\infty (E_i)\rightarrow L^\infty (X_i), \;\;f|_{E_i}\rightarrow f\circ \phi _i, \;\;\;i=1,2
$$ 
are 1-1 unital *-homomorphisms. We define, for $i=1,2$, the TRO's
\begin{equation}\label{defM}
\cl M_i=\{T: TP_i(E)=Q_i(\phi _i^{-1}(E))T, \;\;\text{for}\; \; E\subseteq E_i\;\; \text{Borel}\}
\subseteq\cl B(L^2(E_i), L^2(X_i))
\end{equation}
and we note that $\cl N_i= \cl M_iR_i$ where $R_i\in \cl B(L^2(Y_i))$ 
is the projection onto $L^2(E_i)$, $i=1,2.$

For $i=1,2$, let $\cl A_i\subseteq\cl B(L^2(X_i))$ be the commutant of the commutative 
von Neumann algebra 
$$\{M_{f\circ \phi _i}: \;\;f\in L^\infty (E_i)\}\, .$$
Then it follows from \cite[Theorem 3.2]{ele}  
that $\overline{[\cl M_i^*\cl A_i\cl M_i]}^{w^*}=D(E_i)$ and 
$\overline{[\cl M_iD(E_i)\cl M_i^*]}^{w^*}=\cl A_i$ for $i=1,2$ 
(where $D(E_i)\subseteq \cl B(L^2(E_i))$ 
denotes the multiplication algebra of $L^\infty(E_i)$). 

If $\cl U$ is a $ D(E_2)$-$ D(E_1)$-bimodule  we define  
$$
F(\cl U)=\overline{[\cl M_2\cl U\cl M_1^*]}^{w^*},
$$
where $M_i, i=1,2$ are as in \ref{defM}.

Then,  from \cite[Proposition 2.11]{ele} we obtain that the map 
$$
\cl U\to F(\cl U)
$$ 
maps the family of $ D(E_2)$-$ D(E_1)$-bimodules
contained in $\cl B(L^2(E_1), L^2(E_2))$ bijectively onto the family of $\cl A_2$-$\cl A_1$-bimodules 
contained in $\cl B(L^2(X_1), L^2(X_2)).$ 
The inverse map of $F$ is given by $F^{-1}(V)=\overline{[\cl M_2^*V\cl M_1]}^{w^*}$. 
\smallskip
 
For each masa bimodule  $\cl U\subseteq \cl B(L^2(Y_1), L^2(Y_2))$ we define 
$$F_r(\cl U)=F(R_2\cl UR_1).$$ 

\begin{theorem}\label{equiv} 
If $\kappa\subseteq Y_1\times Y_2$,  is an $\omega$-closed set, then
\begin{align*}
(i) \quad F_r(M_{\max}(\kappa)) &=F(M_{\max}(\kappa\cap (E_1 \times E_2 )))
=M_{\max}((\phi_1 \times \phi _2)^{-1}(\kappa)) \\
\text{and }\; \;
F_r(M_{\min}(\kappa)) &=F(M_{\min}(\kappa\cap (E_1 \times E_2 )))
=M_{\min}((\phi_1 \times \phi _2)^{-1}(\kappa)).
\end{align*}
(ii) 
If the measures $(\phi_i)_*(\mu_i)$ and $\nu_i$ are equivalent for  $i=1,2$, we  have 
\begin{align*}
M_{\max}(\kappa) &=
\overline{[\cl M_2^*M_{\max}((\phi_1 \times \phi _2)^{-1}(\kappa))\cl M_1]}^{w^*} \\
\text{and }\; \;M_{\min}(\kappa) &=
\overline{[\cl M_2^*M_{\min}((\phi_1 \times \phi _2)^{-1}(\kappa))\cl M_1]}^{w^*}.
\end{align*}
\end{theorem}

\proof Part (i) follows from the above discussion, using similar arguments as in 
\cite[Theorem 2.4]{elee}.

For part (ii), note that if the measures are equivalent, then the maps $\tilde\phi_i$ are 
injective, and thus $\nu_i(Y_i\setminus E_i)=0, \;\;\; i=1,2.$ We conclude that 
$$F(M_{\max}(\kappa))=F_r(M_{\max}(\kappa))=M_{\max}((\phi_1 \times \phi _2)^{-1}(\kappa)),$$
and so 
$$M_{\max}(\kappa)=F^{-1}(M_{\max}((\phi_1 \times \phi _2)^{-1}(\kappa))).$$ 
In other words
$$
M_{\max}(\kappa)=\overline{[\cl M_2^*M_{\max}((\phi_1 \times \phi _2)^{-1}(\kappa))\cl M_1]}^{w^*}.
$$
The statement about $M_{\min}(\kappa)$ follows similarly. 
\qed\medskip

For every  $w^*$-closed subspace $V$  of $\cl B(L^2(E_1), L^2(E_2))$ we write 
$$F_n(V\cap \cl K)=\overline{[\cl N_2(V\cap \cl K)\cl N_1^*]}^{\|\cdot\|},$$
where $\cl N_i,  i=1,2$ are as in \ref{defN}.

\begin{theorem}\label{48}
If $\kappa\subseteq Y_1\times Y_2$ is $\omega$-closed, then 
\begin{align*}
F_n(M_{\max}(\kappa)\cap \cl K) &= F_r(M_{\max}(\kappa))\cap \cl K \\
\text{and }\; F_n(M_{\min}(\kappa)\cap \cl K) &= F_r(M_{\min}(\kappa))\cap \cl K.
\end{align*}
\end{theorem} 

\begin{proof} We have 
$$
F_n(M_{\max}(\kappa)\cap \cl K) 
= \overline{[\cl N_2(M_{\max}(\kappa)\cap \cl K )\cl N_1^*]}^{\|\cdot\|} =$$
$$
\overline{[\cl M_2(\cl M_{\max}(\kappa\cap (E_1\times E_2))\cap \cl K )\cl M_1^*]}^{\|\cdot\|}
= F_0(M_{\max}(\kappa\cap (E_1 \times E_2 ))) .$$
By Lemma \ref{47}, the last space  is equal to 
$$F(M_{\max}(\kappa\cap (E_1 \times E_2 ))) 
\cap \cl K=F_r(M_{\max}(\kappa))\cap \cl K.$$
The other equality follows similarly.
\end{proof}

In Theorem \ref{49}  below we improve the inverse image theorem for sets of compact operator 
synthesis obtained by Shulman, Todorov and Turowska in \cite[theorem 4.7]{stwt2}.
Recall that   an  $\omega$-closed set  $\kappa\subseteq Y_1\times Y_2,$ is said to be 
a set of compact operator synthesis
if it satisfies $M_{\max}(\kappa)\cap \cl K=M_{\min}(\kappa)\cap \cl K$.
The sets of compact operator synthesis, as well as their classical counterparts, 
the sets of reduced synthesis   were introduced and studied in \cite{stwt2}. 

In \cite[Theorem 4.7]{stwt2}, the authors proved that,  
 if $\kappa$ is an $\omega$-closed set  of compact operator synthesis,
then $(\phi_1\times \phi_2)^{-1}(\kappa)$ 
is a set of compact operator synthesis under the following assumptions:
The measures $(\phi _i)_*(\mu_i)$ have a Radon-Nikodym derivative with 
respect to $\nu_i$ which are finite a.e. for $i=1,2$ and  the maps 
$\phi_1$ and $\phi_2$ are of a particular form:  the spaces  $X_1$ and $X_2$ 
admit  decompositions $X_1=\tilde X_0\cup\dots\cup \tilde X_m$ and 
$X_2=\tilde Z_0\cup\dots\cup\tilde  Z_l$ such that $\phi_1$ 
is 1-1 when restricted to $\tilde X_0$ and is constant a.e. on each $\tilde X_i, i>0$, 
and similarly for $\phi_2$. 

We arrive at the same conclusion under the assumption that   $\phi _i: X_i\rightarrow Y_i$ are
measurable maps such that $(\phi _i)_*(\mu _i)\ll\nu_i$  for $i=1, 2$.

\begin{theorem}\label{49}
If $\kappa\subseteq Y_1\times Y_2$ is a set of compact operator synthesis then 
$(\phi_1 \times \phi_2)^{-1}(\kappa) $ is also a set of compact operator synthesis. 
\end{theorem}

\begin{proof} If $\kappa$ is a set of compact operator synthesis then by  Theorem \ref{48} we have 
\begin{align*}
F_n(M_{\max}(\kappa)\cap \cl K) &=F_n(M_{\min}(\kappa)\cap \cl K)\\ \text{hence }\;\;
F_r(M_{\max}(\kappa))\cap \cl K &=F_r(M_{\min}(\kappa))\cap \cl K\\ \text{and hence }\;\;
M_{\max}((\phi_1 \times \phi_2 )^{-1}(\kappa))\cap \cl K 
&= M_{\min}((\phi_1 \times \phi_2 )^{-1}(\kappa))\cap \cl K .
\end{align*}
\end{proof}

\begin{theorem}\label{new}  Let $\phi _i: X_i\rightarrow Y_i$ be
measurable maps such that for $i=1,2$ the measures  $(\phi _i)_*(\mu _i)$ and $\nu _i$ 
are equivalent. Then $\kappa\subseteq Y_1\times Y_2$ is a set of compact operator synthesis 
if and only if $(\phi_1 \times \phi_2)^{-1}(\kappa) $ is  a set of compact operator synthesis. 
\end{theorem}

\begin{proof} 
If $\kappa$ is a set of compact operator synthesis then, by Theorem \ref{49}, 
$(\phi_1 \times \phi_2)^{-1}(\kappa) $ 
is also a set of compact operator synthesis. 
{Conversely, assume that   $(\phi_1 \times \phi_2)^{-1}(\kappa)$ 
is a set of compact operator synthesis. By Theorem \ref{equiv}, (ii) and  Lemma \ref{47} } 
$$M_{\max}(\kappa)\cap \cl K=
\overline{[\cl M_2^*(\cl M_{\max}((\phi_1 \times \phi_2)^{-1}(\kappa)\cap \cl K )\cl M_1]}^{\|\cdot\|},$$
$$M_{\min}(\kappa)\cap \cl K
=\overline{[\cl M_2^*(\cl M_{\min}((\phi_1 \times \phi_2)^{-1}(\kappa)\cap \cl K )\cl M_1]}^{\|\cdot\|}$$ 
It follows  that $\kappa$ is a set of compact operator synthesis.
\end{proof}

\section{strict group moprhisms and measures}\label{20000}
We will need the following results: 

Let $G$ be a locally compact group, $N$ a closed normal subgroup of $G$; let  
$G/N$ be the quotient group and $\pi$ the quotient map $G\rightarrow G/N$. 
Choose Haar measures $\mu$, $\lambda$ and $\nu$ on $G$,  $N$ and $ G/N$  respectively
such that, writing $\dot{x}=\pi(x)$, we have
$$\int_Gf(x)d\mu(x)=\int_{G/N}\left(\int_{N}f(xh)d\lambda(h)\right)d\nu(\dot{x})$$
 for every continuous, compactly supported function $f$ on $G$ \cite[Ch. VII, Proposition 10]{bou}. 

Then from \cite[Ch. VII, \S 2, p.57 b)]{bou} we have the following 

\begin{lemma}\label{l21}
Let $G$ be a  second countable locally compact group and $N$ 
 a closed normal subgroup of $G$. If  $\mu$, $\lambda$ and $\nu$ are  Haar 
measures on $G$,  $N$ and $ G/N$  chosen as above, then 
 $$\mu(\pi^{-1}(E))=\lambda (N)\nu (E)$$ for every Borel set 
 $E \subseteq G/N.$
\end{lemma} 

\begin{remark} \label{rem23}
In particular $\pi_* (\mu)$ and $\nu$ are equivalent. 
Note that if $\lambda (N)=+\infty$ and $\nu (E)=0$, second countability ensures that we
have $\mu(\pi^{-1}(E))=0$.
\end{remark}	

If $G, H$ are locally compact groups, $\theta: G\rightarrow H$ is a homomorphism  
and $\mu$ is a measure on $G$, we recall that we
 denote by  $\theta _*(\mu )$ the measure on $H$ defined 
by $\theta _*(\mu )(E)= \mu (\theta^{-1}(E ))$ for  every Borel set $E \subseteq H.$ 

Recall that  $\theta $ is said to be a strict morphism if $\theta (G)$, 
with the relative topology induced by $H$, is homeomorphic to $G/\ker\theta .$

\begin{lemma}\label{22} 
Let $G, H$ be locally compact groups and  $\theta : G\rightarrow H$ be a strict morphism. 
Let  $\mu $, $\nu$ and $\lambda $
be Haar measures on $G$,  $\theta (G)$ and 
$\ker\theta $ respectively. Then 
there exists $0<c<+\infty $ such that 
$$ \mu (\theta^{-1}(E )) =c\lambda (\ker\theta )\nu(E )$$  
for every Borel set $E \subseteq \theta (G).$ 
\end{lemma}

\begin{proof} 
Since $\theta$ is strict, $G/\ker \theta$ is topologically isomorphic to $\theta(G)$.
The assertion follows from Lemma \ref{l21}, by uniqueness of Haar measure on $\ker\theta$. 
\end{proof}
 
\begin{lemma}\label{24} Let $G, H $ be locally compact second countable groups and 
$\theta: G\rightarrow H$ a continuous homomorphism.
Assume that $\theta (G)$ is an open subgroup of $H.$ Then 

(i) $\theta $ is a strict morphism

(ii) $\theta $ is an open map.
\end{lemma} 

\begin{proof}An open subgroup of a topological group is also closed. Hence,
it follows from \cite[Proposition 6]{dr} that $\theta$ is a strict morphism.
The quotient map $G\rightarrow G/\ker {\theta}$
is always open and since $\theta$ is a strict morphism, 
 the induced map $G/\ker {\theta}\to H$ is a homeomorphism. So, $\theta$ is an open map.
\end{proof}

\begin{theorem}\label{25} Let $G, H$ be locally compact second countable groups 
with Haar measures $\mu $ and $\nu $ respectively, and 
$\theta : G\rightarrow H$ be a continuous homomorphism.  The following are equivalent:

(i) $\theta _*(\mu )\ll\nu $

(ii) $\theta(G)$ is an open subgroup of $H$.
\end{theorem}

\begin{proof} (i) $\Rightarrow $ (ii)
If $\nu (\theta (G))=0$ then $\theta _*(\mu )(\theta (G))=\mu (G)=0$ which is absurd. Thus 
$\nu (\theta (G))>0,$ which by Weil's theorem \cite[p. 50]{we} implies that $\theta(G)$ contains an 
open subset, hence  is an open subgroup of $H$. 

(ii) $\Rightarrow $ (i)
 Let $\nu_0$ be the restriction of $\nu$ to   $\theta(G)$. Since $\theta(G)$ 
 is an open subgroup of $H$,  $\nu_0$ is  a Haar measure on $\theta(G)$.
By Lemma \ref{24} $\theta $ is a strict morphism. By Lemma \ref{22}, for some $c>0$ we have 
\begin{equation}\label{*} 
\theta_*(\mu )(E)=\theta _*(\mu )(E\cap\theta (G))= c\lambda (\ker\theta) \nu_0 (E\cap\theta (G)) 
\end{equation}
for every Borel set $E \subseteq H$. Therefore 
$$
\theta_*(\mu )(E )\leq c\lambda (\ker\theta) \nu(E) 
$$
and hence $\theta _*(\mu )\ll\nu .$
\end{proof}

\begin{corollary}\label{26} Let $G, H$ be locally compact second countable groups 
with Haar measures $\mu $ and $\nu $ respectively, and let 
$\theta : G\rightarrow H$ be a continuous homomorphism
such that  $\theta _*(\mu )\ll\nu .$ Then 

(i) If $\ker\theta $ is compact then $\theta _*(\mu )$ is a Haar measure for $\theta (G).$

(ii) If $\ker\theta $ is not compact then  $\theta _*(\mu )(E )\in \{0,+\infty \}$ 
for every Borel set $E \subseteq \theta (G).$
\end{corollary}

\begin{proof} From the proof of Theorem \ref{25} it follows that there exists $0<c<+\infty$   such that 
$$
\mu(G)=\theta_*(\mu )(\theta(G))\le c\lambda (\ker\theta) \nu(\theta(G)) 
$$
and hence $\gl(\ker\theta)>0$. Therefore, using (\ref{*}), 
$$\theta_*(\mu)(E)=c\lambda (\ker\theta) \nu (E )$$ for  every Borel set $E \subseteq \theta (G).$
 This fact implies the conclusion.
\end{proof} 

\section{Synthesis and Compact operator synthesis }\label{sec4}

Recall that the Fourier-Stieltjes algebra of $G$ is the set of all coefficient functions 
$s\to\sca{\pi(s)\xi,\eta}, \ (\xi,\eta\in H_\pi)$ defined by unitary representations 
$(\pi, H_\pi)$ of $G$, while the Fourier algebra $A(G)$ of $G$ consists of the coefficients 
of the left regular representation $s\to\gl_s$ on $L^2(G)$, given by $\gl_s\xi(t):=\xi(s^{-1}t)$.

We study preservation and transference 
properties  of  continuous group homomorphisms between locally compact groups.
These homomorphisms preserve important objects of Harmonic Analysis. For example, if 
$\theta : G\rightarrow H$ is a continuous homomorphism, then the map  
$u\mapsto u\circ \theta $  is a contraction from  the Fourier algebra $A(H)$ of 
$H$ into the Fourier-Stieltjes algebra $B(G)$ of $G$ \cite[Th\'{e}or\`{e}me 2.20(1)]{eym}. 

If $J\subseteq A(G)$ is an ideal, we denote by  $Z(J)$  the set of all points of $G$ 
where all $u\in J$ vanish. If $E$ is   a closed subset of $G$,  we set 
\begin{align*}
I_G(E)=I(E) & =\{u\in A(G) : u(s)=0,\,  s\in E\}\, , \\
J_G(E) = J(E) &=\overline{\{u \in A(G) : u \text{ has compact support disjoint from } E\}}\, .
\end{align*}	
Then 
$$Z(J(E)) = Z( I(E)) = E$$
and, if $J\subseteq A(G)$ is a closed ideal with $Z(J) = E$, then $J(E)\subseteq J\subseteq I(E)$.

A closed subset $E\subseteq G$ is called a {\it set of spectral synthesis} if $I(E) = J(E)$.
Shulman, Todorov and Turowska introduce and study in \cite{stwt2} the notions 
of reduced spectral synthesis and reduced local spectral synthesis for closed subsets 
of a locally compact group $G$. A closed set $E$ is a set of {\it reduced spectral synthesis} if
$C_r^*(G)\cap I(E)^\perp = C_r^*(G)\cap J(E)^\perp$. 
It is a set of {\it reduced local spectral synthesis} if 
$C_r^*(G)\cap I_c(E)^\perp = C_r^*(G)\cap J(E)^\perp$, where 
$I_c(E)$ denotes the functions in  $I(E)$ of compact support.  
Here $C_r^*(G)$ is the reduced $C^*$-algebra of $G$, that is, the $C^*$-subalgebra of 
$\cl B(L^2(G))$ generated by all operators $\gl(f), f\in L^1(G)$, where   
$\gl(f)(h)= f*h,  \ h\in L^2(G)$. For motivation, examples and relevant discussion 
about these concepts,  see \cite[Section 3]{stwt2}.

In what follows, for a subset $E\subseteq G$ we write
$$ E^*=\{(s,t)\in G\times G: \;\;ts^{-1}\in E\}\, .$$

\begin{theorem}\label{411} 
Let $G, H$ be locally compact second countable groups 
with Haar measures $\mu $ and $\nu $ 
respectively and let $\theta : G\rightarrow H$ be a continuous homomorphism such that 
$\theta (G)$ is an open set in $H.$ 
If $E\subseteq H$ is a set of reduced local spectral synthesis then $\theta ^{-1}(E)$ 
is also a set of  reduced local spectral synthesis. 
\end{theorem}

\begin{proof} If $E\subseteq H$ is a set of reduced local spectral synthesis 
then \cite[Theorem 5.1]{stwt2} implies that $E^*$ 
is a set of compact operator synthesis. By Theorem \ref{49}, $  (\theta \times \theta )^{-1}( E^*)  $ 
is also a set of compact operator synthesis. Since 
$\theta ^{-1}(E)^*= (\theta \times \theta )^{-1}( E^*)$, using \cite[Theorem 5.1]{stwt2} again 
we conclude that $\theta ^{-1}(E)$  is a set of reduced local spectral synthesis.
\end{proof}

Under the assumptions of Theorem \ref{411}, observe that   
the measures $\theta_*(\mu)$ and $\nu$ are equivalent if and only if $\theta$ 
is onto. Indeed, if $\theta$  is onto, by Theorem \ref{25} and Corollary \ref{26} we have that 
$$\mu(\theta^{-1}(A))=c \nu(A)$$
for every Borel set $A\subseteq H,$ where $c$ is a positive constant, perhaps infinity if  
$\ker \theta$ is not compact. 
Thus $\theta_*(\mu)$ and $\nu$ are equivalent. 
Conversely, assume that $\theta_*(\mu)$ and $\nu$ are equivalent.
Then, since $\theta_*(\mu)(H\setminus \theta(G))=0$ we have $\nu(H\setminus \theta(G))=0.$
But every open subgroup is also closed, so $H\setminus \theta(G)$  
is open; since it is $\nu$-null, it must be empty. 
Thus $\theta(G)=H.$
 
Under the above assumptions the map 
$$\tilde \theta : L^\infty(H)\rightarrow L^\infty(G), \; f\rightarrow f\circ \theta$$  
is an injective $*$-homomorphism. Indeed assume that $\tilde \theta$ is not injective then  
$\ker \tilde\theta= L^\infty(A)$ for a Borel set $A\subseteq H$ with positive measure $\nu(A).$
Thus, $\mu(\theta^{-1}(A))$  is positive. This implies that $\tilde \theta(\chi_A)\neq 0$ 
which is a contradiction.

\smallskip 
In the following result, the `only if' direction is due to Shulman and Turowska \cite[Theorem 4.7]{st}.  
Recall  that a subset $\kappa\subseteq H\times H$ is $\omega$-closed if its complement is marginally
equivalent to a countable union of Borel rectangles. 
 
\begin{theorem}\label{oc4} 
Let $G, H$ be locally compact second countable groups 
with Haar measures $\mu$ and $\nu$ respectively, 
$\theta$ be  a continuous onto homomorphism from $G$ to $H,$ and 
$\kappa\subseteq H\times H$ be an $\omega$-closed set. 
Then the set $\kappa$ is a set of operator synthesis if and only if 
$(\theta\times\theta)^{-1}(\kappa)$  is a set of operator synthesis. 
\end{theorem}
  
\begin{proof}   
We define the TRO 
$$
\cl M=\{T\in\cl B(L^2(H), L^2(G)): TP(E)
=Q(\theta^{-1}(E))T, \;\;\text{for all }\; \; E\subseteq H\;\; \text{Borel}\},$$
where $P(E)$ is the projection onto $L^{2}(E, \nu)$ and $Q(\theta^{-1}(E))$ 
is the projection onto $L^{2}(\theta^{-1}(E), \mu).$ 
Since the map $\tilde \theta : L^\infty(H)\to L^\infty(G): f\to f\circ \theta$
is an injective $*$-homomorphism, from Theorem \ref{equiv} we have that 
$$
M_{\max}((\theta\times\theta)^{-1}(\kappa))=\overline{[\cl M M_{\max}(\kappa) \cl M^{*}]}^{w^*}, $$
$$
M_{\max}(\kappa)=\overline{[\cl M^{*}( M_{\max}(\theta\times\theta)^{-1}(\kappa)) \cl M]}^{w^*}, $$
$$
M_{\min}((\theta\times\theta)^{-1}(\kappa))=\overline{[\cl M M_{\min}(\kappa) \cl M^{*}]}^{w^*}, $$
$$
M_{\min}(\kappa)=\overline{[\cl M^{*}( M_{\min}(\theta\times\theta)^{-1}(\kappa)) \cl M]}^{w^*}, $$
By the definition of operator synthesis, the result follows.
\end{proof}
 
In \cite[Theorems 4.3 and 4.11]{lt} the authors show that a closed set $E\subseteq G$ is a set of 
local synthesis if and only if $E^{*}\subseteq G\times G$ is a set of operator synthesis. Therefore, 
using the above theorem, we obtain the following result which is proved by Lohou\'{e} in \cite{lo}.  
 
\begin{corollary}\label{oc5} 
Let $G, H$ be locally compact second countable groups, $\theta$ 
be  a continuous onto homomorphism from $G$ to $H,$ and $E\subseteq H$ 
be a closed set. Then $E$ is a set of local synthesis  if and only if $\theta^{-1}(E)$ 
is a set of local synthesis.  
\end{corollary}

\begin{remark}
 Note that, conversely, one could prove 
 that the conclusion of  theorem \ref{oc4} holds for subsets of $H \times H$ of
 the form $E^*$, where $E$ is a closed subset of $H$,  using 
 the result of Lohou\'{e} and  the above mentioned results of \cite{lt}.
 \end{remark}
 
\begin{theorem}\label{oc2} 
Let $G, H$ be locally compact second countable groups 
with Haar measures $\mu$ and $\nu$  respectively, $\theta$ be  a continuous 
onto homomorphism from $G$ to $H,$ and $\kappa\subseteq H\times H$ 
be an $\omega$-closed set. The set $\kappa$ is a set of compact operator synthesis  
if and only if $(\theta\times\theta)^{-1}(\kappa)$ is a set of compact operator synthesis.
\end{theorem}   

\begin{proof} In this case the measures $\theta_*(\mu)$ and $\nu$ are equivalent and the conclusion follows from  Theorem \ref{new}. 
\end{proof}
  

\section{Ditkin sets} \label{ditkin}

Operator Ditkin sets were first defined and studied by Shulman and Turowska in \cite{st}.
Ludwig and Turowska introduced the notion of strong operator Ditkin sets  in \cite{lt}.
In this section we prove an inverse image theorem for strong operator Ditkin sets.
 
Recall that, for a locally compact group $G$,  
the predual of $\cl B(L^2(G))$ may be identified with the space $T(G)$ 
of (equivalence classes of marginally a.e equal) 
functions $h: G\times G\to \mathbb C$ for which there exist sequences 
$(f_n)$ and $(g_n)$ in $L^2(G)$ such that 
$$
h(x,y)=\sum_{n=1}^\infty f_n(x)g_n(y), \;\;\sum_{n=1}^\infty \|f_n\|_2\|g_n\|_2<+\infty .
$$
The norm of such a function is 
$$
\|h\|_{T(G)}=\inf{\sum_{n=1}^\infty f_n(x)g_n(y)}
$$  
where the infimum is taken over all such sequences for which 
$$
h(x,y)=\sum_{n=1}^\infty f_n(x)g_n(y) 
$$
marginally a.e..
The duality between $T(G)$ and $\cl B(L^2(G))$ is given by the pairing 
$$
 \langle T,h\rangle := \sum_{n=1}^{\infty} \sca{Tf_n,\bar g_n}
$$
\cite[p. 494]{arv}.

Let $\mathfrak{S}(G)$ be the multiplier algebra of $T(G)$; by definition, a
measurable function $w : G\times G\rightarrow \bb{C}$ belongs to $\mathfrak{S}(G)$ if
the map $m_w: h\to wh$ leaves $T(G)$ invariant (that is, if
$wh$ is marginally equivalent to a function from $T(G)$,
for every $h\in T(G)$) and defines a bounded map on $T(G)$ \cite{spronk}
(see also \cite[Proposition 4.9]{ivan} where it is shown that the continuity of $m_w$ is automatic). 
The elements of $\mathfrak{S}(G)$ are called {\it (measurable) Schur multipliers}.

Fix locally compact second countable groups $G$ and $H$ with Haar measures 
$\mu$ and $\nu$ respectively. 

We will assume that the Fourier algebra $A(G)$ satisfies Ditkin's 
condition at infinity (or $D_\infty$), namely, that every $u\in A(G)$ belongs to 
$\overline{uA(G)}$. See Remark \ref{Dit} below.

We also assume that $\theta: G\rightarrow H$ is a continuous homomorphism 
such that $\ker \theta\subseteq G$ is compact  and $\theta(G)\subseteq H$ 
is open. Since $\theta(G)$ is an open subgroup, by Corollary \ref{26},
$\theta_*(\mu)$ is a Haar measure for $\theta(G)$, 
so (multiplying $\nu$ by a positive constant if necessary) we may assume that 
$\nu|_{\theta(G)}=\theta_*(\mu).$  
Observe that the map $L^2(H)\to L^2(G): f\mapsto f\circ \theta$ is contractive;
hence the map
$$T(H)\rightarrow T(G):\;\; h\rightarrow h\circ (\theta\times\theta)$$ 
is a contraction.

If $\kappa\subseteq G\times G$ is an $\omega$-closed set, we write  
$$\Psi(\kappa) = \{h\in T(G): h\chi_{\kappa} = 0 \mbox{ marginally a.e.}\}\, .$$
(This denoted by $\Phi_0(\kappa)$ in \cite{st}.)
By \cite[Theorem 4.3]{st},
\begin{equation}\label{eq_minmax}
M_{\min}(\kappa) = \Psi(\kappa)^{\perp} 
\end{equation}

Recall (see \cite{lt}) that an $\omega$-closed set $\kappa\subseteq G\times G$ 
is called a {\it strong operator Ditkin set} 
if there exists a sequence $(w_n)$ in $\mathfrak S(G)$ such that each $w_n$ vanishes in an 
$\omega$-neighbourhood of $\kappa$ and 
$$\|w_n h-h\|_{T(G)}\rightarrow 0,\;\;\text{for all }\;\;h\;\in\;\Psi(\kappa).$$ 

Let $E\subseteq H$ be a closed set such that $E^*$ is a strong operator Ditkin set. 
It follows from \cite[Theorem 5.4]{lt}  that $E$ is  a set of local spectral synthesis. 
Recall that a closed subset $E\subseteq G$ is called a {\it set of local spectral synthesis} if 
$J(E)$ contains every compactly supported function of $I(E)$.

 Theorem 5.3 in \cite{akt} implies that 
$$M_{\min}(\theta^{-1}(E)^*)={\rm Bim}(I_G( \theta^{-1}(E) )^\perp)$$ 
and  Corollary 2.5 in \cite{elee} implies that
 $\theta^{-1}(E)^*$ is a set of operator synthesis. Thus, 
$$ M_{\max}(\theta^{-1}(E)^*) =M_{\min}(\theta^{-1}(E)^*).$$
Now, if $\rho(u)=u\circ \theta$ for all $u\in A(H)$, from Theorem 3.4 in \cite{elee} we have that 
$$M_{\max}(\theta^{-1}(E)^*)={\rm Bim}( \langle\rho(I_H(E)\rangle ^\perp),$$
where $\langle\rho(I_H(E)\rangle $ is the closed ideal in $A(G)$ generated by $\rho(I_H(E)).$ Thus
$${\rm Bim}( I_G(\theta^{-1}(E))^\perp)={\rm Bim}( \langle\rho(I_H(E)\rangle^\perp).$$ 
Since $A(G)$ satisfies $D_\infty$, Lemma 4.5 in \cite{akt} implies that 
$$I_G(\theta^{-1}(E)=\langle\rho(I_H(E)\rangle.$$
Thus $$\Psi((\theta\times\theta)^{-1}(E^*))^\bot={\rm Bim}( \langle\rho(I_H(E)\rangle^\perp).$$

\begin{lemma}\label{lem47} 
The space $$\Psi((\theta\times\theta)^{-1}(E^*))$$ 
is  a subspace of  the masa-bimodule generated in $T(G)$  by the set 
$$ \{h\circ(\theta \times \theta): \;h\;\in \;\Psi(E^*)\} .$$
\end{lemma} 

\begin{proof} We saw that   the space 
$ \Psi((\theta\times\theta)^{-1}(E)^*)$ is the  preannihilator
of ${\rm Bim}( \langle\rho(I_H(E)\rangle^\perp)$. 
By \cite[Theorem 3.2]{akt} this preannihilator  is  
$$\overline{{\rm span} \{N(\rho(u))h: u\in I_H(E),\;\;h\in T(G) \}}^{\|\cdot \|_{T(G)}}$$ 
where $N(f)(s,t):=f(ts^{-1})$ (see \cite[Proposition 3.1]{akt}). 

 Let $u\in I_H(E)$.  The restriction  $u|_{\theta(G)}$ of $u$ to $\theta(G)$ 
belongs to $A(\theta(G))$ \cite[Proposition 3.21(2)]{eym} and 
the function $u|_{\theta(G)}\circ \theta$ belongs to $A(G)$ \cite[Proposition 3.25(1)]{eym}. 
Since $u|_{\theta(G)}\circ \theta=u\circ \theta$, 
 $\rho(u)=u\circ \theta$ is  in $A(G)$ and therefore
$N(u)\circ(\theta\times\theta)$ vanishes m.a.e. on $(\theta\times\theta)^{-1}(E^*)$. 
It follows that  for all $h\in T(G)$, the function   
$(N(u)\!\circ\!(\theta\times\theta))h$ belongs to the masa-bimodule generated by the set 
$$\{k\!\circ\!(\theta \times \theta): \;k\;\in \;\Psi(E^*)\} .$$ 
\end{proof}

\begin{remark}\label{Dit}
Note that    Lemma 4.5 of [2], which we have used, 
remains true   if  instead of assuming that $A(G)$ has an approximate identity, we assume 
that $A(G)$ satisfies the formally weaker Ditkin's condition at infinity. 
This follows immediately from its proof. 
Ditkin's condition at infinity is a `local' condition, closer in spirit to reflexivity 
(as defined by Loginov-Shulman) rather than the `global' condition of an approximate identity.  
However  it is unknown if there are any locally compact groups failing either condition.
 We note that a characterization of  Ditkin's condition at infinity related to our discussion 
 is given by Andreou \cite{DA}.
\end{remark}

\begin{theorem}\label{th48} 
Let $G$ and $H$ be locally compact second countable groups 
such that $A(G)$ satisfies $D_\infty$ and 
let $\theta: G\rightarrow H$ be a continuous homomorphism 
such that $\ker \theta\subseteq G$ is compact  and $\theta(G)\subseteq H$ 
is open. Let $E\subseteq H$ be a closed subset. If $E^*$ is a strong operator Ditkin set, then 
$(\theta^{-1}(E))^*$ is also a strong operator Ditkin set.
\end{theorem}

\begin{proof} 
Writing $H_1=\theta (G)$ for brevity, we can easily see that 
$$\Psi((E\cap \theta(G))^*)=\zeta_0 \Psi(E^*),$$ 
where $\zeta_0=\chi_{H_1\times H_1}$. Since $E^*$ is an operator Ditkin set, there exists 
a sequence  $(w_n)\subseteq\mathfrak S (H)$ such that $w_n$ vanishes in an 
$\omega$-neighbourhood of $E^*$ and 
$$\|w_n h-h\|_{T(H)}\rightarrow 0,\;\;\text{for all }\;\;h\;\in\;\Psi(E^*).$$
Thus 
$$
\|w_nh\zeta_0-h\zeta_0\|_{T(H_1)}\rightarrow 0,
\;\;\text{for all }\;\;h\;\in\;\Psi(E^*)$$
and so 
 $$\|w_nh_1-h_1\|_{T(H_1)}\rightarrow 0,
 \;\;\text{for all }\;\;h_1\;\in\;\Psi((E\cap H_1)^*).$$
Since the map 
$$T(H_1)\rightarrow T(G): h_1\rightarrow h_1\circ (\theta \times\theta)$$
is a contraction we have 
$$
\|(w_n\circ (\theta \times\theta)(h_1\circ (\theta \times\theta))
-h_1\circ (\theta \times\theta))\|_{T(G)}\rightarrow 0,\;\;\text{for all }\;\;h_1\;\in\;\Psi((E\cap H_1)^*)
$$
and, a fortiori,
$$
\|(w_n\circ (\theta \times\theta)(h_1\circ (\theta \times\theta))w
-h_1\circ (\theta \times\theta))w\|_{T(G)}\rightarrow 0,$$
for all $h_1\;\in\;\Psi((E\cap H_1)^*)$ and $w\in\mathfrak S(G)$.
But now if $h_2\in \Psi((E\cap H_1)^*)=\Psi( (\theta \times\theta)^{-1}(E^*))$ 
then, by Lemma \ref{lem47}, $h_2$ is in the masa bimodule  
$\{h\circ(\theta\times\theta)w: h\in \Psi(E^*), w\in\mathfrak S(G)\}$ 
(since masa bimodules are invariant under Schur multipliers) and so we have 
$$
\|(w_n\circ (\theta\times\theta))h_2-h_2\|_{T(H)}\rightarrow 0,\;\;
\text{for all }\;\;h_2\;\in\;\Psi((\theta\times\theta)^{-1}(E^*)).
$$
Since  $w_n\circ (\theta\times\theta)$ vanishes in an $\omega$-neighbourhood of 
$(\theta\times\theta)^{-1}(E^*) = (\theta^{-1}(E))^*$, we conclude that 
$(\theta^{-1}(E))^*$ is a strong operator Ditkin set. 
 \end{proof}
 
 \begin{remark}
Note, for comparison, that Parthasarathy and Kumar have proved in \cite{pk} that if $G$ 
is a compact group and $H$ a normal subgroup  of $G$, then a closed subset $E$ of $G/H$ 
is a strong Ditkin set if and only if $\pi^{-1}(E)$ is a strong Ditkin set, 
where $\pi: G\rightarrow G/H$ is the quotient map.
 \end{remark}
 
\begin{remark}\label{auto} 
The above theorem together with Theorem 5.4 in \cite{lt} implies the following: 
Let $G$ and $H$ be locally compact second countable groups such that 
$A(G)$ satisfies  $D_\infty$ and let $\theta: G\rightarrow H$ 
be a continuous homomorphism 
such that $\ker \theta\subseteq G$ is compact  and $\theta(G)\subseteq H$ 
is open. Let $E\subseteq H$ be a closed subset. If $E^*$ is a strong operator Ditkin set then 
$\theta^{-1}(E)$ is a local Ditkin set.
\end{remark}


\section{M-sets}\label{msets}
The concept of $M$-sets for general locally compact groups  was introduced by Bo\.{z}ejko in \cite{bo}.
Shulman, Todorov and Turowska   introduced in \cite{stwt} the notion of 
$M_1$ sets in locally compact groups   and studied transference and preservation 
properties of $M$-sets and $M_1$-sets  and their operator analogues, 
which had been  introduced in \cite{stwtint}.

A closed subset $E\subseteq G$ is called a {\it set  of multiplicity} (or an {\it $M$-set})
if $C_r^*(G)\cap J(E)^\perp \ne \{0\}$. 
The set $E$ is called an {\it $M_1$-set} if $C_r^*(G)\cap I(E)^\perp \ne \{0\}$.

An $\omega$-closed set $\kappa \subseteq G \times G$ is called an {\it operator $M$-set} 
if $M_{\max}(\kappa)$ contains nonzero compact operators, and is called 
an {\it operator $M_1$-set} if $M_{\min}(\kappa)$ contains nonzero compact operators.
\smallskip

In the following theorem we show that a set $\kappa$ is an operator M-set if and only if 
$(\theta\times\theta)^{-1}(\kappa)$ is an operator M-set. 
The `only if' direction of this result is proved in \cite{stwt} 
under the  assumption that $\theta_*(\mu)$ has a Radon-Nikodym derivative which is 
$\mu$-a.e. finite, which in turn is equivalent to 
the compactness of $\ker\theta$ (see Corollary \ref{26}).

\begin{theorem}\label{oc2ii} 
Let $G, H$ be locally compact second countable groups with Haar measures 
$\mu$ and $\nu$  respectively,  $\theta$ be  a continuous 
onto homomorphism from $G$ to $H,$ and $\kappa\subseteq H\times H$ 
be an $\omega$-closed set. 
 
 Then the set $\kappa$ is an operator $M$-set  (resp. $M_1$-set)  if and only if 
$(\theta\times\theta)^{-1}(\kappa)$ is an operator $M$-set  (resp. $M_1$-set).  
\end{theorem}

\begin{proof}  
As we noted in section \ref{sec4},  the condition $\theta (G)=H$ is equivalent 
to the requirement that the measures $\theta_*(\mu)$ and $\nu$ are equivalent, 
even if $\theta_*(\mu)$ is not $\sigma$-finite. 

We define the TRO 
$$
\cl M=\{x: xP(E)
=Q(\theta ^{-1}(E))x, \;\;\text{for all }\; \; E\subseteq H\;\; \text{Borel}\},
$$
where $P(E)$ is the projection onto $L^{2}(E, \nu)$ and $Q(\theta^{-1}(E))$ 
is the projection onto $L^{2}(\theta^{-1}(E), \mu).$ 
 From Theorem \ref{equiv} 
and Lemma  \ref{47} we have that
$$
M_{\max}((\theta\times\theta)^{-1}(\kappa))\cap \cl K 
=\overline{[\cl M( M_{\max}(\kappa)\cap \cl K) \cl M^{*}]}^{\|\cdot\|}, 
$$
$$
M_{\max}(\kappa)\cap \cl K 
=\overline{[\cl M^{*}( M_{\max}((\theta\times\theta)^{-1}(\kappa))\cap \cl K) \cl M]}^{\|\cdot\|}, 
$$
$$
M_{\min}((\theta\times\theta)^{-1}(\kappa))\cap \cl K 
=\overline{[\cl M( M_{\min}(\kappa)\cap \cl K) \cl M^{*}]}^{\|\cdot\|}, 
$$
$$
M_{\min}(\kappa)\cap \cl K
=\overline{[\cl M^{*}( M_{\min}((\theta\times\theta)^{-1}(\kappa))\cap \cl K) \cl M]}^{\|\cdot\|}. 
$$
If $\kappa$ is an operator $M$-set then $M_{\max}(\kappa)$ contains a non-zero compact operator 
$x$.  We claim that $\cl Mx\cl M^{*}\neq \{0\}.$ 
Indeed: assume to the contrary that $\cl Mx\cl M^{*}=\{0\}.$ Then 
$\overline{[\cl M^{*}\cl M]}^{w^*}x\overline{[\cl M^{*}\cl M]}^{w^*}=\{0\}.$ 
By \cite[Lemma 3.1]{ele} the algebra $\overline{[\cl M^{*}\cl M]}^{w^*}$ 
contains the identity operator and thus $x=0.$ This contradiction shows that  
$\cl Mx\cl M^{*}\neq \{0\}$ which implies that 
$M_{\max}((\theta\times\theta)^{-1}(\kappa))\cap \cl K\neq \{0\}.$
Therefore the set $(\theta\times\theta)^{-1}(\kappa)$ is an operator $M$-set. 

Reversing the roles of $\cl M$ and $\cl M^*$, we show that if $(\theta\times\theta)^{-1}(\kappa)$ 
is an operator $M$-set then $\kappa$ is an operator $M$-set. 

Similarly we can prove that $\kappa$ is an operator $M_1$-set  
if and only if $(\theta\times\theta)^{-1}(\kappa)$ is an operator $M_1$-set.  
\end{proof}
  
In \cite{stwt}, the authors showed that if $\theta: G\rightarrow H$ is 
a  continuous homomorphism onto $H$ and $E$ is a closed subset of $H$ which  is an $M$-set 
(resp. an  $M_1$-set) then  $\theta^{-1}(E)$ is also an $M$-set (resp. an  $M_1$-set)
under the  assumption that $\theta_*(\mu)$ has a Radon-Nikodym derivative which is 
$\mu$-a.e. finite, which in turn is equivalent to 
the compactness of $\ker\theta$ (see Corollary \ref{26}). 

In \cite{elee} it was proved that if $\theta: G\rightarrow H$ is 
an open continuous homomorphism and $\theta^{-1}(E)$ is an $M$-set 
(resp. an  $M_1$-set) then $E$ is also an $M$-set (resp. an  $M_1$-set). 

It follows from  \cite{stwt}  that a closed set $E\subseteq G$ is an $M$-set 
(resp. an $M_1$-set) if and only if $E^{*}\subseteq G\times G$ is an operator $M$-set 
(resp. an operator $M_1$-set). Hence  Theorem \ref{oc2ii} implies   Corollary \ref{oc3} 
below, which  for $M$-sets  is due to Derighetti \cite{der}.
 We note that conversely, one could prove 
 that the conclusion of  theorem \ref{oc2ii} holds for subsets of $H \times H$ of
 the form $E^*$, where $E$ is a closed subset of $H$,  using 
 the result of Derighetti and  the above mentioned result of \cite{stwt}.
 
\begin{corollary}\label{oc3} 
Let $G, H$ be locally compact second countable groups, $\theta$ 
be  a continuous onto homomorphism from $G$ to $H,$ and $E\subseteq H$ be a closed set. 
Then $E$ is an $M$-set  (resp. an $M_1$-set)  if and only if $\theta^{-1}(E)$ is an $M$-set  
(resp. an $M_1$-set).  
\end{corollary}

\begin{remark} \label{sets} 
In case $\theta$ is not onto, Corollary \ref{oc3} does not necessarily hold even if 
$\theta$ is an open map with compact kernel. 
For example take $G=\mathbb Z_2$ and 
$$\theta: G\rightarrow G, \;\theta(x)=0,\;\;\text{for all } \;x\in G.$$
Consider the set $E=\{1\}.$ This is an $M$-set and an $M_1$-set, but $\theta^{-1}(E)=\emptyset$ 
is not an $M$-set or an $M_1$-set.
\end{remark}

\end{document}